\documentclass{amsart}
\usepackage{amsmath}
\usepackage{amssymb}
\usepackage{amsthm}
\usepackage[msc-links,abbrev]{amsrefs}
\usepackage{hyperref}
\usepackage{enumitem}
\usepackage{mathrsfs}
\usepackage{mathtools}
\usepackage{slashed}
\usepackage[noabbrev,capitalize]{cleveref}

\setlist[enumerate]{parsep=0ex}
\setlist[enumerate,1]{label=(\roman*)}

\DeclareMathOperator{\dvol}{dvol}

\DeclareMathOperator{\Imaginary}{Im}
\DeclareMathOperator{\Real}{Re}

\newcommand{\defn}[1]{{\boldmath\bfseries#1}}

\newcommand{\oZ}{\overline{Z}}

\newcommand{\oz}{\overline{z}}

\newcommand{\cM}{\widetilde{M}}

\newcommand{\cT}{\widetilde{T}}

\newcommand{\cxi}{\widetilde{\xi}}

\newcommand{\lp}{\langle}
\newcommand{\rp}{\rangle}
\newcommand{\lv}{\lvert}
\newcommand{\rv}{\rvert}

\newcommand{\bCP}{\mathbb{C}P}

\newcommand{\dbbar}{\overline{\partial}_b}

\newcommand{\mQ}{\mathcal{Q}}

\newcommand{\bC}{\mathbb{C}}

\newcommand{\bN}{\mathbb{N}}

\newcommand{\bZ}{\mathbb{Z}}

\newcommand{\sP}{\mathscr{P}}

\def\sideremark#1{\ifvmode\leavevmode\fi\vadjust{\vbox to0pt{\vss
 \hbox to 0pt{\hskip\hsize\hskip1em
 \vbox{\hsize3cm\tiny\raggedright\pretolerance10000
 \noindent #1\hfill}\hss}\vbox to8pt{\vfil}\vss}}}

\newcommand{\suchthat}{\mathrel{}\middle|\mathrel{}}
\newcommand{\suchthatcolon}{\mathrel{}:\mathrel{}}

\newtheorem{theorem}{Theorem}[section]

\newtheorem{lemma}[theorem]{Lemma}
\newtheorem{corollary}[theorem]{Corollary}

\theoremstyle{definition}
\newtheorem{definition}[theorem]{Definition}

\theoremstyle{remark}
\newtheorem{remark}[theorem]{Remark}

\numberwithin{equation}{section}

\usepackage{tikz-cd}
\usetikzlibrary{arrows}

\begin{document}

\title{A CR invariant sphere theorem}
\author{Jeffrey S. Case}
\address{Department of Mathematics \\ Penn State University \\ University Park, PA 16802 \\ USA}
\email{jscase@psu.edu}
\author{Paul Yang}
\address{Department of Mathematics \\ Princeton University \\ Princeton, NJ 08540 \\ USA}
\email{yang@math.princeton.edu}
\begin{abstract}
 We prove that every closed, universally embeddable CR three-manifold with nonnegative Yamabe constant and positive total $Q^\prime$-curvature is contact diffeomorphic to a quotient of the standard contact three-sphere.
 We also prove that every closed, embeddable CR three-manifold with zero Yamabe constant and nonnegative total $Q^\prime$-curvature is CR equivalent to a compact quotient of the Heisenberg group with its flat CR structure.
\end{abstract}
\maketitle

\section{Introduction}
\label{sec:intro}

The Gauss--Bonnet Theorem states that if $(M^2,g)$ is a closed surface, then
\begin{equation*}
 2\pi\chi(M) = \int_M Q^g\,\dvol_g ,
\end{equation*}
where $\chi(M)$ is the Euler characteristic of $M$ and $Q^g$ is the Gauss curvature of $g$.
In particular, the nonnegativity of the conformal invariant $\int Q$ has strong topological implications:
\begin{enumerate}
 \item If $\int Q>0$, then $M$ is diffeomorphic to a quotient of the two-sphere.
 \item If $\int Q=0$, then $(M,g)$ is conformal to a two-torus.
\end{enumerate}
Of course, $\int Q$ actually determines the genus, and hence the topological type, of the universal cover of $M$.

On even-dimensional Riemannian manifolds, Branson's $Q$-curvature~\cite{Branson1995} is a good generalization of the Gauss curvature.
For example, on closed manifolds, the total $Q$-curvature $\int Q^g\,\dvol_g$ is conformally invariant~\cites{Branson1995,GrahamZworski2003}, proportional to the Euler characteristic on locally conformally flat manifolds~\cite{BransonGilkeyPohjanpelto1995}, and a linear combination of the Euler characteristic and integrals of local conformal invariants on general manifolds~\cite{Alexakis2012}.
The total $Q$-curvature also has some strong topological implications~\cites{Gursky1994,ChangGurskyYang2003,GuanViaclovskyWang2003}.
The most relevant to us is the sharp conformally invariant sphere theorem of Chang, Gursky and Yang~\cite{ChangGurskyYang2003}:

\begin{theorem}
 \label{cgy}
 Let $(M^4,g)$ be a closed Riemannian manifold with nonnegative Yamabe constant.
 If $\int Q \geq \int \lv W\rv^2$, then
 \begin{enumerate}
  \item $M$ is diffeomorphic to a quotient of the sphere;
  \item $(M,g)$ is conformal to a quotient of $(S^1\times S^3,g_{prod})$;
  \item $(M,g)$ is conformal to $(\bCP^2,g_{FS})$; or
  \item $(M,g)$ is conformal to a flat torus.
 \end{enumerate}
\end{theorem}

Here $W$ is the Weyl tensor --- which vanishes if and only if $(M,g)$ is locally conformally flat --- and the Yamabe constant $Y(M,[g])$ is a global conformal invariant which is positive (resp.\ zero) if and only if $g$ is conformal to a metric with positive (resp.\ zero) scalar curvature~\cite{LeeParker1987}.
In the last three cases of \cref{cgy} it holds that $\int Q = \int \lv W\rv^2$.
Moreover, in the second case of \cref{cgy}, $\int Q = 0$ and $Y(M,[g])>0$, and $g_{prod}$ denotes the product of the flat metric on $S^1$ with the the round metric on $S^3$;
in the third case of \cref{cgy}, $\int Q = 16\pi^2$ and $Y(M,[g])>0$, and $g_{FS}$ denotes the Fubini--Study metric;
and in the fourth case, $\int Q = 0$ and $Y(M,[g])=0$.

In this article we establish a partial CR analogue of \cref{cgy}.
There is a deep analogy between CR three-manifolds and conformal four-manifolds which suggests that the nonnegativity of the CR Yamabe constant~\cite{JerisonLee1987}, the CR Paneitz operator~\cite{GrahamLee1988}, and the total $Q^\prime$-curvature~\cites{CaseYang2012,Takeuchi2019} implies that the underlying CR three-manifold is either
\begin{enumerate}
 \item contact diffeomorphic to a quotient of the standard contact three-sphere;
 \item CR equivalent to a quotient of $S^1 \times S^2$ with its standard CR structure; or
 \item CR equivalent to a compact quotient of the Heisenberg group.
\end{enumerate}
We verify a weaker version of this classification by imposing either the vanishing of the CR Yamabe constant --- in which case the underlying CR three-manifold is CR equivalent to a compact quotient of the Heisenberg group --- or by imposing the stronger assumptions that the CR Yamabe constant and total $Q^\prime$-curvatures are positive, and that all finite covers of the CR manifold are embeddable.
To make this precise requires some terminology.

A \defn{CR three-manifold} is a pair $(M^3,T^{1,0})$ of a three-manifold $M^3$ and a complex rank one distribution $T^{1,0} \subset \bC TM$ such that
\begin{enumerate}
 \item $T^{1,0} \cap T^{0,1} = \{ 0\}$ for $T^{0,1} := \overline{T^{1,0}}$, and
 \item if $Z$ is a nowhere-vanishing local section of $T^{1,0}$, then $\bigl\{ Z, \oZ, [Z,\oZ] \bigr\}$ locally generates $\bC TM$.
\end{enumerate}
Set $\xi := \Real(T^{1,0}\oplus T^{0,1})$.
Let $\theta$ be a local \defn{contact form};
i.e.\ a local real one-form with $\ker\theta=\xi$.
The condition that $\bigl\{ Z, \oZ, [Z,\oZ] \bigr\}$ generates $\bC TM$ implies that $\theta \wedge d\theta$ is nowhere vanishing.
In particular, $(M,\xi)$ is a \defn{contact three-manifold}.
Note that the CR structure $T^{1,0}$ orients $\xi$, and hence $(M,\xi)$ admits a global contact form.
We say that a contact form $\theta$ is \defn{positive} if $i\theta([Z,\oZ])>0$ for any nowhere-vanishing local section $Z$ of $T^{1,0}$.

The \defn{standard CR three-sphere} $(S^3,T^{1,0})$ is the unit sphere $S^3 \subset \bC^2$ with $T^{1,0} := \bC TS^3 \cap T^{1,0}\bC^2$.
The \defn{standard contact three-sphere} $(S^3,\xi)$ is the corresponding contact manifold.
Note that $S^3$ supports many contact structures;
the standard contact structure is the only one which arises by realizing $S^3$ as the boundary of a symplectic manifold~\cite{Eliashberg1990}.

A CR manifold $(M^3,T^{1,0})$ is \defn{embeddable} if there is an embedding $\iota \colon M \to \bC^N$, $N \gg 1$, such that $\iota_\ast T^{1,0} \subset T^{1,0}\bC^N$.
This is a global property of $(M^3,T^{1,0})$, and is equivalent to the requirement that the $\dbbar$-operator has closed range~\cite{Kohn1986}.
Note that if $(M^3,T^{1,0})$ is embeddable, then $(M,\xi)$ is tight~\cites{Eliashberg1990,Gromov1985}.
We say that $(M^3,T^{1,0})$ is \defn{universally embeddable} if every finite cover of $(M^3,T^{1,0})$ is embeddable.
In particular, if $(M,T^{1,0})$ is universally embeddable, then $(M^3,\xi)$ is universally tight.
Not all embeddable CR three-manifolds are universally embeddable;
for example, there is an embeddable $\bZ_2$-quotient of the non-embeddable Rossi spheres~\cite{ChenShaw2001}.
Indeed, there are many lens spaces with tight contact structure which are not universally tight~\cites{Giroux2000,Honda2000}.

A \defn{pseudohermitian manifold} is a triple $(M^3,T^{1,0},\theta)$ consisting of a CR manifold and a choice of positive contact form.
This gives rise to a unique connection~\cites{Tanaka1975,Webster1978} which preserves $T^{1,0}$ and $\theta$.
The curvature of this connection determines a scalar pseudohermitian invariant, the \defn{Webster curvature}, which is analogous to the usual scalar curvature.
The CR Yamabe constant $Y(M,T^{1,0})$, introduced by Jerison and Lee~\cite{JerisonLee1987}, is a global CR invariant of $(M^3,T^{1,0})$ which is positive (resp.\ zero) if and only if there is a positive contact form $\theta$ with positive (resp.\ zero) scalar curvature.
The CR Yamabe constant is the CR analogue of the Yamabe constant.

A pseudohermitian manifold is \defn{$Q$-flat} if Hirachi's $Q$-curvature~\cite{Hirachi1990} vanishes.
If $(M^3,T^{1,0})$ is closed and embeddable, then the space of $Q$-flat contact forms is infinite-dimensional and parameterized by the space of CR pluriharmonic functions~\cite{Takeuchi2019}.
The \defn{$Q^\prime$-curvature}~\cites{CaseYang2012,Hirachi2013} is a pseudohermitian invariant whose total integral is independent of the choice of $Q$-flat contact form on a given closed, embeddable CR three-manifold~\cites{CaseYang2012,Takeuchi2019}.
Importantly, the total $Q^\prime$-curvature is a nontrivial invariant;
it equals $16\pi^2$ on the standard CR three-sphere, and more generally coincides~\cite{CaseYang2012} with the Burns--Epstein invariant~\cite{BurnsEpstein1988} on boundaries of strictly pseudoconvex domains in $\bC^2$.
The total $Q^\prime$-curvature is the CR analogue of the total $Q$-curvature.

We can now precisely state our CR analogues of \cref{cgy}.

First, we have the following CR sphere theorem for closed, universally embeddable CR three-manifolds with nonnegative CR Yamabe constant and positive total $Q^\prime$-curvature.

\begin{theorem}
 \label{sphere}
 Let $(M^3,T^{1,0})$ be a closed, universally embeddable CR manifold.
 If $\int Q^\prime>0$ and $Y(M,T^{1,0}) \geq 0$, then $(M^3,\xi)$ is contact diffeomorphic to a quotient of the standard contact three-sphere.
\end{theorem}

\Cref{sphere} has a number of nice properties.
First, its assumptions and conclusion are all CR invariant;
i.e.\ they are independent of the choice of pseudohermitian contact form.
Second, its proof implies that $\lv \pi_1(M) \rv \leq 16\pi^2/\int Q^\prime$.
In particular, \cref{sphere} implies a CR invariant gap theorem characterizing the standard contact three-sphere in terms of the total $Q^\prime$-curvature;
see \cref{detailed-positive-q}.
While we do not know if either of these statements are true if the universal embeddability assumption is relaxed to embeddability, they are true for the $\bZ_2$-quotients of the Rossi spheres.
Third, there are embeddable deformations of the standard CR structure on $S^3$ which descend to lens spaces.
In particular, the universal embeddability assumption does not trivialize \cref{sphere} by determining the CR structure.
See \cref{sec:positive-q} for details.

Second, we have the following classification of closed, embeddable CR three-manifolds with zero CR Yamabe constant and nonnegative total $Q^\prime$-curvature.
Note that we do not require the stronger assumption of universal embeddability. 

\begin{theorem}
 \label{flat}
 Let $(M^3,T^{1,0})$ be a closed, embeddable CR manifold.
 If $\int Q^\prime = 0$ and $Y(M,T^{1,0})=0$, then $(M^3,T^{1,0})$ is CR equivalent to a compact quotient of the Heisenberg group.
\end{theorem}

As previously noted, a closed CR three-manifold is embeddable if and only if the $\dbbar$-operator has closed range~\cite{Kohn1986}.
Another characterization of embeddability can be given in terms of the \defn{CR Paneitz operator}~\cites{Hirachi1990,GrahamLee1988}, which is a fourth-order, CR invariant, formally self-adjoint operator whose kernel contains the CR pluriharmonic functions.
Explicitly, a closed, CR three-manifold with positive CR Yamabe constant is embeddable if and only if the CR Paneitz operator is nonnegative~\cite{Takeuchi2019}.
We say that $(M^3,T^{1,0})$ has \defn{universally nonnegative CR Paneitz operator} if the CR Paneitz operator of all of its finite covers is nonnegative.
This yields a reformulation of \cref{sphere}:

\begin{corollary}
 \label{paneitz-sphere}
 Let $(M^3,T^{1,0})$ be a closed CR manifold with positive CR Yamabe constant, universally nonnegative CR Paneitz operator, and positive total $Q^\prime$-curvature.
 Then $(M^3,\xi)$ is contact diffeomorphic to a quotient of the standard contact three-sphere.
\end{corollary}

The proof of \cref{sphere} differs from the proof of the analogous statement in \cref{cgy} in that we do not proceed by choosing a contact form for which $Q^\prime$, or some related scalar pseudohermitian invariant, is constant.
Instead, we observe that an upper bound~\cite{CaseYang2012} for the total $Q^\prime$-curvature in terms of the CR Yamabe constant of a closed, embeddable CR three-manifold implies an upper bound on the degree of any finite cover of $(M^3,T^{1,0})$.
The topological conclusion then follows from the resolution of the Poincar\'e Conjecture~\cites{Perelman1,Perelman2,Perelman3} and the classification~\cites{Eliashberg1990} of tight contact structures on $S^3$.
We expect that an alternative proof which relies on choosing an appropriate contact form (cf.\ \cite{CaseHsiaoYang2014}) will allow one to relax the universal embeddability assumption.

The proof of \cref{flat} is analogous to the proof of the corresponding statement in \cref{cgy}.

This article is organized as follows.
In \cref{sec:bg} we give some additional background on CR and pseudohermitian geometry in dimension three.
In \cref{sec:zero-yamabe} we prove \cref{flat}.
In \cref{sec:positive-q} we prove the stronger version of \cref{sphere} which estimates $\pi_1(M)$ in terms of the total $Q^\prime$-curvature.
We also give examples of non-spherical, universally embeddable CR structures on lens spaces, and comment on the total $Q^\prime$-curvature of the ($\bZ_2$-quotients of the) Rossi spheres.

\section{Some CR geometry}
\label{sec:bg}

In this section we discuss in more detail the important facts about CR and pseudohermitian three-manifolds mentioned in the introduction.

Let $(M^3,T^{1,0},\theta)$ be a pseudohermitian manifold.
The \defn{Reeb vector field} is the unique vector field $T$ such that $\theta(T)=1$ and $d\theta(T,\cdot)=0$.
An \defn{admissible coframe} is a local, nowhere-vanishing, complex-valued one-form $\theta^1$ such that $\theta^1(T)=0$ and $\theta^1(\oZ)=0$ for all local sections $Z$ of $T^{1,0}$.
It follows that there is a positive function $h_{1\bar 1}$ such that
\begin{equation*}
 d\theta = ih_{1\bar 1}\,\theta^1\wedge\theta^{\bar1} ,
\end{equation*}
where $\theta^{\bar1} := \overline{\theta^1}$.
Given an admissible coframe $\theta^1$, there is~\cite{Webster1978} a unique connection one-form $\omega_1{}^1$ such that
\begin{align*}
 d\theta^1 & \equiv \theta^1 \wedge \omega_1{}^1 \mod \theta \wedge \theta^{\bar 1} , \\
 \omega_{1\bar1} + \omega_{\bar 11} & = dh_{1\bar 1} ,
\end{align*}
where $\omega_{1\bar1} := h_{1\bar1}\omega_1{}^1$ and $\omega_{\bar11} := \overline{\omega_{1\bar1}}$.
This determines the \defn{Tanaka--Webster connection} by
\begin{equation*}
 \nabla Z_1 := \omega_1{}^1 \otimes Z_1 \quad\text{and}\quad \nabla T := 0 ,
\end{equation*}
where $Z_1$ is the unique local section of $T^{1,0}$ such that $\theta^1(Z_1)=1$.
The \defn{torsion} of $\theta$ is the (globally-defined) section $A_{11}\,\theta^1\otimes\theta^1$, where $A_{11} = A_1{}^{\bar1}h_{1\bar1}$ is determined by
\begin{equation*}
 d\theta^1 = \theta^1 \wedge \omega_1{}^1 + A_1{}^{\bar1}\theta \wedge \theta^{\bar1} .
\end{equation*}
Consider the curvature two-form $\Omega_1{}^1 := d\omega_1{}^1$.
The \defn{Webster curvature} is the (globally-defined) real-valued function $R$ determined by
\begin{equation*}
 \Omega_1{}^1 \equiv Rh_{1\bar1}\,\theta^1\wedge\theta^{\bar 1} \mod \theta .
\end{equation*}

There are three important CR covariant operators of relevance to this paper.

The first operator is the \defn{CR Yamabe operator} $L_\theta \colon C^\infty(M) \to C^\infty(M)$,
\begin{equation*}
 L_\theta(u) := -\Delta_bu + \frac{R}{4}u ,
\end{equation*}
where $\Delta_bu := \nabla^1\nabla_1u + \nabla_1\nabla^1u$.
This operator is formally self-adjoint.
It is also CR covariant:
If $v \in C^\infty(M)$ is positive, then~\cite{JerisonLee1987}*{Equation~(3.1)}
\begin{equation*}
 v^3L_{v^2\theta}(u) = L_\theta\left( uv\right)
\end{equation*}
for all $u \in C^\infty(M)$.
It follows that if $M$ is compact, then the \defn{CR Yamabe constant}
\begin{equation*}
 Y(M,T^{1,0}) := \inf \left\{ \int_M uL_\theta(u)\,\theta \wedge d\theta \suchthatcolon \int_M \lv u\rv^4\,\theta \wedge d\theta = 1 \right\}
\end{equation*}
is CR invariant.
If $(M,T^{1,0})$ is embeddable, then~\cite{Takeuchi2019}*{Theorem~1.4} there is a \defn{CR Yamabe contact form};
i.e.\ a contact form $\theta$ such that $\int \theta \wedge d\theta=1$ and $R^\theta = 4Y(M,T^{1,0})$.
This follows from existence results~\citelist{ \cite{JerisonLee1988}*{Corollary~B} \cite{JerisonLee1987}*{Theorem~3.4(c)} } on the standard CR three-sphere and on CR manifolds with $Y(M,T^{1,0}) < \pi$ together with a sharp estimate on the CR Yamabe constant:

\begin{lemma}
 \label{cr-yamabe-estimate}
 Let $(M^3,T^{1,0})$ be a closed, embeddable CR manifold.
 It holds that $Y(M,T^{1,0}) \leq \pi$ with equality if and only if $(M,T^{1,0})$ is CR equivalent to the standard CR three-sphere.
\end{lemma}

\begin{proof}
 Let $(M^3,T^{1,0})$ be a closed, embeddable CR manifold.
 Jerison and Lee proved~\citelist{\cite{JerisonLee1987}*{Theorem~3.4(b)} \cite{JerisonLee1989}*{Corollary~B}} that $Y(M,T^{1,0}) \leq Y(S^3,T^{1,0}) = \pi$.
 Cheng, Malchiodi and Yang~\cite{ChengMalchiodiYang2013}*{Theorem~1.1} showed that if the CR Paneitz operator is nonnegative, then equality holds if and only if $(M^3,T^{1,0})$ is CR equivalent to the standard CR three-sphere.
 Takeuchi later removed~\cite{Takeuchi2019}*{Theorem~1.1} the assumption on the CR Paneitz operator.
\end{proof}

The second operator is the \defn{CR Paneitz operator} $P_\theta \colon C^\infty(M) \to C^\infty(M)$,
\begin{equation*}
 P_\theta(u) := 4\nabla^1(\nabla_1\nabla_1 + iA_{11})\nabla^1u .
\end{equation*}
The CR Paneitz operator is a real-operator~\cite{GrahamLee1988}*{p.\ 710} --- i.e.\ $\overline{P(u)} = P(u)$ --- and hence formally self-adjoint.
It is also CR covariant~\cite{Hirachi1990}*{Lemma~7.4}:
\begin{equation*}
 e^{2\Upsilon}P_{e^\Upsilon\theta}(u) = P_\theta(u)
\end{equation*}
for all $u,\Upsilon \in C^\infty(M)$.
Recall that a function $u \in C^\infty(M)$ is \defn{CR pluriharmonic} if locally there is an $f \in C^\infty(M;\bC)$ such that $u = \Real f$ and $\nabla_{\bar1}f=0$.
The space $\sP$ of CR pluriharmonic functions is characterized~\cite{Lee1988}*{Proposition~3.4} as
\begin{equation*}
 \sP := \left\{ u \in C^\infty(M) \suchthatcolon (\nabla_1\nabla_1 + iA_{11})\nabla^1u = 0 \right\} .
\end{equation*}
In particular, $\sP \subseteq \ker P$.
Moreover, if $(M^3,T^{1,0})$ is closed and embeddable, then equality holds~\cite{Takeuchi2019}*{Theorem~1.1}.

The third operator is the \defn{$P^\prime$-operator} $P^\prime \colon \sP \to C^\infty(M)$,
\begin{equation*}
 \begin{split}
  P^\prime(u) & := 4\Delta_b^2u - 8\Imaginary \nabla^1(A_{11}\nabla^1u) - 4\Real \nabla^1(R\nabla_1u) \\
   & \quad + \frac{8}{3}\Real W_1\nabla^1u - \frac{4}{3}u\nabla^1W_1  ,
 \end{split}
\end{equation*}
where
\begin{equation*}
 W_1 := \nabla_1R - i\nabla^1A_{11} .
\end{equation*}
This operator is formally self-adjoint~\citelist{ \cite{Hirachi2013}*{Theorem~4.5} \cite{CaseYang2012}*{Proposition~4.6} }.
It transforms as a $Q$-curvature operator (cf.\ \cites{Case2021q,BransonGover2005}) under change of contact form:
If $\Upsilon \in C^\infty(M)$, then~\cite{CaseYang2012}*{Proposition~4.6}
\begin{equation*}
 e^{2\Upsilon}P_{e^\Upsilon\theta}^\prime(u) = P_\theta^\prime(u) + P_\theta(\Upsilon u)
\end{equation*}
for all $u \in \sP$.

The \defn{$Q$-curvature}~\cite{Hirachi1990} of a pseudohermitian manifold $(M^3,T^{1,0},\theta)$ is
\begin{equation*}
 Q_\theta := -\frac{4}{3}\nabla^1W_1 .
\end{equation*}
Its name reflects its CR transformation formula:
If $\Upsilon \in C^\infty(M)$, then
\begin{equation*}
 e^{2\Upsilon}Q_{e^\Upsilon\theta} = Q_\theta + P_\theta \Upsilon .
\end{equation*}
If $(M,T^{1,0})$ is the boundary of a closed, strictly pseudoconvex domain in $\bC^2$, then the Fefferman defining function~\cite{Fefferman1976} gives rise to a $Q$-flat contact form~\cite{FeffermanHirachi2003}.
More generally, all closed, embeddable CR three-manifolds admit $Q$-flat contact forms.

\begin{theorem}[\cite{Takeuchi2019}*{Theorem~1.5}]
 \label{takeuchi-existence}
 Let $(M^3,T^{1,0})$ be a closed, embeddable CR manifold.
 Then there is a $Q$-flat contact form $\theta$ on $(M^3,T^{1,0})$.
 Moreover, $e^\Upsilon\theta$ is $Q$-flat if and only if $\Upsilon \in \sP$.
\end{theorem}

Let $(M^3,T^{1,0},\theta)$ be a pseudohermitian manifold.
The \defn{$Q^\prime$-curvature} is
\begin{equation*}
 Q_\theta^\prime := -2\Delta_bR - 4\lv A_{11}\rv^2 + R^2 .
\end{equation*}
Then~\cite{CaseYang2012}*{Proposition~6.1}
\begin{equation}
 \label{eqn:q-prime-general-transformation}
 \begin{split}
  e^{2\Upsilon}Q_{e^\Upsilon\theta}^\prime & = Q^\prime + P^\prime \Upsilon + \frac{16}{3}\Real \nabla^1(\Upsilon W_1) + 3Q \Upsilon \\
   & \quad + \frac{1}{2}P_\theta(\Upsilon^2) - \Upsilon P_4 \Upsilon - 16\Real(\nabla^1 \Upsilon)(\nabla_1\nabla_1 + iA_{11})\nabla^1 \Upsilon .
 \end{split}
\end{equation}
\Cref{eqn:q-prime-general-transformation} underlies the interpretation~\cite{CaseYang2012} of the $Q^\prime$-curvature as the CR analogue of the $Q$-curvature.
In one direction, \cref{eqn:q-prime-general-transformation,takeuchi-existence} imply~\cites{CaseYang2012,Takeuchi2019} that the total $Q^\prime$-curvature of a closed, embeddable CR three-manifold is independent of the choice of $Q$-flat contact form.

\begin{definition}
 Let $(M^3,T^{1,0})$ be a closed, embeddable CR manifold.
 The \defn{total $\boldsymbol{Q^\prime}$-curvature} is
 \begin{equation*}
  \mQ^\prime(M,T^{1,0}) := \int_M Q_\theta^\prime \, \theta \wedge d\theta ,
 \end{equation*}
 where $\theta$ is any $Q$-flat contact form on $(M^3,T^{1,0})$.
\end{definition}

In another direction, \cref{eqn:q-prime-general-transformation} gives an estimate for the total $Q^\prime$-curvature of a closed, embeddable CR three-manifold in terms of $\int Q_\theta^\prime\,\theta \wedge d\theta$ for $\theta$ an arbitrary contact form.

\begin{lemma}
 \label{q-prime-no-q-flat}
 Let $(M^3,T^{1,0},\theta)$ be a closed, embeddable, pseudohermitian manifold.
 Then
 \begin{equation*}
  \mQ^\prime(M,T^{1,0}) \leq \int_M Q_\theta^\prime \, \theta \wedge d\theta
 \end{equation*}
 with equality if and only if $\theta$ is $Q$-flat.
\end{lemma}

\begin{proof}
 Let $\theta_0$ be a $Q$-flat contact form and write $\theta = e^\Upsilon\theta_0$.
 \Cref{eqn:q-prime-general-transformation} implies that
 \begin{equation*}
  \int_M Q_\theta^\prime \, \theta \wedge d\theta = \mQ^\prime(M,T^{1,0}) + 3\int_M \Upsilon\,P\Upsilon .
 \end{equation*}
 Since $(M^3,T^{1,0})$ is embeddable, it holds~\cite{Takeuchi2019}*{Theorem~1.1} that $\int \Upsilon\,P\Upsilon \geq 0$ with equality if and only if $\Upsilon \in \sP$.
 The conclusion readily follows from \cref{takeuchi-existence}.
\end{proof}

Applying \cref{q-prime-no-q-flat} to a CR Yamabe contact form gives a particularly useful estimate on the total $Q^\prime$-curvature (cf.\ \cite{CaseYang2012}*{Theorem~1.1}).

\begin{corollary}
 \label{q-prime-upper-bound}
 Let $(M^3,T^{1,0})$ be a closed, embeddable CR manifold with nonnegative CR Yamabe constant.
 Then
 \begin{equation*}
  \mQ(M,T^{1,0}) \leq 16\pi^2
 \end{equation*}
 with equality if and only if $(M^3,T^{1,0})$ is CR equivalent to the standard CR three-sphere.
\end{corollary}

\begin{proof}
 Let $\theta$ be a CR Yamabe contact form.
 Then
 \begin{equation}
  \label{eqn:estimate-q-prime-from-yamabe}
  \int_M Q_\theta^\prime \, \theta \wedge d\theta \leq 16Y(M,T^{1,0})^2
 \end{equation}
 with equality if and only if $\theta$ is torsion-free.
 Since $Y(M,T^{1,0})\geq0$, the conclusion follows immediately from \cref{cr-yamabe-estimate}.
\end{proof}

\section{Proof of \cref{flat}}
\label{sec:zero-yamabe}

The classification of closed, embeddable CR three-manifolds with zero CR Yamabe constant and nonnegative total $Q^\prime$-curvature follows easily from \cref{eqn:estimate-q-prime-from-yamabe}.

\begin{proof}[Proof of \cref{flat}]
 Let $(M^3,T^{1,0})$ be a closed, embeddable CR manifold with $Y(M^3,T^{1,0})=0$ and $\mQ^\prime(M,T^{1,0})\geq0$.
 Let $\theta$ be a CR Yamabe contact form.
 Combining \cref{q-prime-no-q-flat} with \cref{eqn:estimate-q-prime-from-yamabe} and its characterization of equality implies that $\theta$ is torsion-free.
 The result now follows from the classification~\cite{Tanno1969}*{Proposition~4.1} of closed CR three-manifolds which admit a torsion-free contact form of zero Webster curvature.
\end{proof}

\section{The case $\mQ^\prime(M,T^{1,0}) > 0$}
\label{sec:positive-q}

The main idea to handle the case of positive total $Q^\prime$-curvature is that \cref{q-prime-upper-bound} gives an upper bound on the degree of any finite cover of a closed, universally embeddable CR manifold.

\begin{lemma}
 \label{cover-estimate}
 Let $(M^3,T^{1,0})$ be a closed, universally embeddable CR manifold with $Y(M,T^{1,0}) \geq 0$ and $\mQ^\prime(M,T^{1,0})>0$.
 Then any finite connected cover of $M$ has degree at most $16\pi^2/\mQ^\prime(M,T^{1,0})$.
\end{lemma}

\begin{proof}
 Let $\pi \colon \cM^3 \to M^3$ be a finite cover of degree $k$.
 Set
 \begin{equation*}
  \cT^{1,0} := \left\{ Z \in \bC T\cM \suchthat \pi_\ast Z \in T^{1,0} \right\} .
 \end{equation*}
 Then $(\cM^3,\cT^{1,0})$ is a closed, embeddable CR three-manifold.
 
 Since $Y(M,T^{1,0})\geq0$, there is a contact form $\theta$ on $(M,T^{1,0})$ with nonnegative Webster curvature.
 Therefore $\pi^\ast\theta$ is a contact form on $(\cM,\cT^{1,0})$ with nonnegative Webster curvature.
 It follows~\cite{JerisonLee1987}*{Lemma~6.4} that $Y(\cM,\cT^{1,0})\geq0$.
 
 Now let $\theta_0$ be a $Q$-flat contact form on $(M,T^{1,0})$.
 Then $\pi^\ast\theta_0$ is a $Q$-flat contact form on $(\cM,\cT^{1,0})$.
 Since $\pi$ is a $k$-fold covering, it holds that
 \begin{equation*}
  \mQ^\prime(\cM,\cT^{1,0}) = \int_{\cM} Q_{\pi^\ast\theta}^\prime\,\pi^\ast\theta \wedge d(\pi^\ast\theta) = k\mQ(M^3,T^{1,0}) .
 \end{equation*}
 We conclude from \cref{q-prime-upper-bound} that $k\mQ(M,T^{1,0}) \leq 16\pi^2$.
\end{proof}

Recall that the fundamental group $\pi_1(M)$ of a manifold $M$ is \defn{residually finite} if for each non-identity element $x \in \pi_1(M)$, there is a homomorphism $\varphi \colon \pi_1(M) \to G$ onto a finite group $G$ such that $\varphi(x) \not= e$.
This homomorphism determines a finite connected cover $\pi \colon \cM \to M$ of degree $\lv G\rv$.
By iterating this procedure, we see that if $\pi_1(M)$ is infinite, then for each $N \in \bN$ there is a finite connected cover $\pi \colon \cM \to M$ of degree at least $N$.

We now prove \cref{sphere}.
Indeed, we prove a sharper result which uses \cref{cover-estimate} to bound the size of the fundamental group of a closed, embeddable CR manifold with nonnegative CR Yamabe constant and positive total $Q^\prime$-curvature.

\begin{theorem}
 \label{detailed-positive-q}
 Let $(M^3,T^{1,0})$ be a closed, universally embeddable CR manifold with $Y(M,T^{1,0}) \geq 0$ and $\mQ^\prime(M,T^{1,0})>0$.
 Then the underlying contact manifold $(M^3,\xi)$ is contact diffeomorphic to a quotient $(S^3/\Gamma,q_\ast\xi)$ of the standard contact three-sphere.
 Moreover, $\lv \Gamma \rv \leq 16\pi^2 / \mQ^\prime(M^3,T^{1,0})$.
\end{theorem}

\begin{proof}
 Recall~\cite{AschenbrennerFriedlWilton2015}*{Item~(C.29)} that every closed three-manifold has residually finite fundamental group.
 We readily deduce from \cref{cover-estimate} that $\pi_1(M)$ is finite.
 
 Let $(\cM,\cT^{1,0})$ be the universal cover of $(M,T^{1,0})$.
 The resolution of the Poincar\'e Conjecture~\cite{MorganTian2007}*{Theorem~0.1} implies that $\cM$ is diffeomorphic to $S^3$.
 \Cref{cover-estimate} then implies that $M \cong S^3/\Gamma$ for some finite group $\Gamma$ such that
 \begin{equation*}
  \lv \Gamma \rv \leq \frac{16\pi^2}{\mQ^\prime(M^3,T^{1,0})} .
 \end{equation*}
 Since $(\cM,\cT^{1,0})$ is embeddable, the underlying contact structure $\cxi := \Real(\cT^{1,0} \oplus \cT^{0,1})$ is tight~\cite{Geiges2008}*{Theorem~6.5.6}.
 The classification of tight contact structures on the three-sphere~\cite{Geiges2008}*{Theorem~4.10.1(a)} then implies that $(\cM,\cT^{1,0})$ is contact diffeomorphic to the standard contact three-sphere.
\end{proof}

\Cref{detailed-positive-q} immediately implies a CR invariant gap theorem.

\begin{corollary}
 Let $(M^3,T^{1,0})$ be a closed, universally embeddable CR manifold with $Y(M,T^{1,0}) \geq 0$ and $\mQ^\prime(M,T^{1,0})>8\pi^2$.
 Then $(M^3,\xi)$ is contact diffeomorphic to the standard contact three-sphere.
\end{corollary}

We conclude with a discussion of the universal embeddability hypothesis in \cref{sphere,detailed-positive-q}.
First, we construct nonspherical CR structures on lens spaces.

\begin{lemma}
 \label{lens-space}
 Given integers $p > q > 0$, define $\Gamma \colon S^3 \to S^3$ by
 \begin{equation*}
  \Gamma(z_1,z_2) := (\omega z_1 , \omega^q z_2) ,
 \end{equation*}
 where $\omega$ is a $p$-th root of unity.
 There is an $\varepsilon>0$ such that for all $t \in (-\varepsilon,\varepsilon)$, the CR structure $T_t^{1,0}$ on $S^3$ generated by
 \begin{equation*}
  Z_t := (\oz_1 \partial_{z_2} - \oz_2 \partial_{z_1}) + t\oz_1^{2q+2} (z_1 \partial_{\oz_1} - z_2 \partial_{\oz_1})
 \end{equation*}
 descends to a universally embeddable CR structure on $L(p,q) := S^3 / \lp \Gamma \rp$, where $\lp \Gamma \rp$ is the group generated by $\Gamma$.
 Moreover, if $t \not= 0$, then $\bigl(L(p,q),T_t^{1,0}\bigr)$ is not locally spherical.
\end{lemma}

\begin{proof}
 First note that $2q+2 \geq 4$.
 Therefore the CR manifolds $(S^3,T_t^{1,0})$ are embeddable for $t$ sufficiently close to zero~\cite{BurnsEpstein1990}*{Theorem~5.3}.
 Next observe that
 \begin{equation*}
  \Gamma_\ast Z_t = \omega^{q+1}Z_t .
 \end{equation*}
 In particular, $\Gamma$ preserves the bundle $T_t^{1,0}$.
 Since $\Gamma$ is a diffeomorphism, we conclude that it is a CR automorphism of $(S^3,T_t^{1,0})$.
 Since $\Gamma$ has no fixed points and $\lp \Gamma \rp$ is discrete, we conclude that $T_t^{1,0}$ descends to a CR structure, still denoted $T_t^{1,0}$, on $L(p,q)$.
 The final conclusion follows from a computation~\cite{CurryEbenfelt2021}*{Lemma~6.2} of the linearization of the Cartan tensor.
\end{proof}

The lens spaces considered in \cref{lens-space} are all universally tight.
However, there are lens spaces which admit tight, but not universally tight, contact structures~\citelist{ \cite{Giroux2000}*{Th\'eor\`eme~1.1} \cite{Honda2000}*{Theorem~2.1 and Proposition~5.1} }.

Second, we point out that the total $Q^\prime$-curvature of the Rossi spheres is bounded above by $16\pi^2$, with equality if and only if the CR structure is locally spherical.
Since the Rossi spheres have positive CR Yamabe constant~\cite{ChanilloChiuYang2010}, we see that the $\bZ_2$-quotients of the Rossi spheres are consistent with \cref{detailed-positive-q}.

\begin{lemma}
 \label{rossi}
 Let $(S^3,T_t^{1,0})$, $t \in (-1, 1)$, be the Rossi sphere, where $T_t^{1,0}$ is spanned by
 \begin{equation*}
  Z_t := (\oz_2 \partial_{z_1} - \oz_1 \partial_{z_2}) + t (z_2 \partial_{\oz_1} - z_1 \partial_{\oz_2})
 \end{equation*}
 Then
 \begin{enumerate}
  \item $Y(S^3,T_t^{1,0}) > 0$; and
  \item $\int Q^\prime \leq 16\pi^2$ with equality if and only if $t=0$.
 \end{enumerate}
\end{lemma}

\begin{remark}
 For $\lv t\rv$ small, Cheng, Malchiodi and Yang~\cite{ChengMalchiodiYang2019}*{Theorem~1.2} proved that $Y(S^3,T_t^{1,0}) = \pi$, showing that the embeddability assumption of \cref{cr-yamabe-estimate} is necessary.
 In light of their observation, \cref{rossi} suggests that the universal embeddability assumption of \cref{sphere} might be weakened to embeddability.
\end{remark}

\begin{proof}
 Set $\theta := i\,(z_1 \, d\oz_1 + z_2 \, d\oz_2)$.
 Then~\cite{ChanilloChiuYang2010}*{Proposition~2.5} the Webster curvature and $Q^\prime$-curvature of $(S^3,T_t^{1,0},\theta)$ are
 \begin{align*}
  R & = \frac{2(1+t^2)}{1-t^2} , \\
  Q^\prime & = \frac{4(1-14t^2 + t^4)}{(1-t^2)^2} ,
 \end{align*}
 respectively.
 In particular, $Q^\prime \leq 4$ with equality if and only if $t=0$.
 The final conclusion follows from the readily-verified fact $\int \theta \wedge d\theta = 4\pi^2$.
\end{proof}

\section*{Acknowledgements}
We thank John Pardon for a helpful discussion about residual finiteness of the fundamental groups of closed three-manifolds.
JSC would like to thank the University of Washington for providing a productive research environment while the authors were completing this work.
JSC was partially supported by the Simons Foundation (Grant \# 524601).
PY was partially supported by the Simons Foundation (Grant \# 615589).

\bibliography{bib}
\end{document}